\documentclass[12pt]{amsart}
\usepackage{amsmath,amscd}
\usepackage{amsthm}
\usepackage{amsfonts}
\usepackage{amssymb}
\usepackage{tikz}
\usepackage{mathtools}
\usetikzlibrary{positioning}
\usepackage{graphicx,rotating}
\usepackage{esint}
\usepackage{float}
\usepackage{hyperref}
\usepackage{enumerate}
\usepackage{blkarray}

\title[$A_\f{q}$-components of geometric classes]{$A_\f{q}$-components of geometric classes in compact Hermitian locally symmetric spaces}
\author{Arghya Mondal}
\address{School of Mathematics, Tata Institute of Fundamental Research, Homi Bhaba Road, Navy Nagar, Colaba, Mumbai 400005, India}
\email{arghya@math.tifr.res.in}
\date{}

\newtheorem{theorem}{Theorem}[section]
\newtheorem{lemma}[theorem]{Lemma}

\newtheorem{cor}[theorem]{Corollary}

\newenvironment{definition}[1][Definition]{\begin{trivlist}
\item[\hskip \labelsep {\bfseries #1}]}{\end{trivlist}}
\newenvironment{remark}[1][Remark]{\begin{trivlist}
\item[\hskip \labelsep {\bfseries #1}]}{\end{trivlist}}

\def\f{\mathfrak}
\def\r{\mathbb{R}}
\def\c{\mathbb{C}}

\def\n{\mathbb{N}}

\def\G{\Gamma}

\def\b{\backslash}

\begin{document}

\maketitle

\begin{abstract}
Let $\G\b G/K$ be a compact Hermitian locally symmetric space, where $G$ is simple. We study the components of a de Rham cohomology class of $\G\b G/K$, with respect to the Matsushima decomposition, where the class is obtained by taking Poincar\'e dual of a totally geodesic complex analytic submanifold. Using an extension of the vanishing result of Kobayashi and Oda, we specify the existence of certain components of such cohomology classes when $G=\text{SU}(p,q), 5\le p\le q$.
%or $\text{SO}_0(2,4)$.  
\end{abstract}

\noindent\textit{All real Lie algebras and their subspaces will be denoted with a $0$ subscript. Dropping of this subscript will denote the complexification.}

\section{Introduction}
Let $G$ be a linear semisimple Lie group with no compact normal subgroups. Let $K$ be a maximal compact subgroup of $G$. Let $\G$ be a torsion-free uniform lattice in $G$. Then $\G\b G/K$ is a compact locally symmetric space. Let $\f{g}_0$ and $\f{k}_0$ be the Lie algebras of $G$ and $K$ respectively and let $\theta$ be the Cartan involution of $\f{g}_0$, which fixes $\f{k}_0$ pointwise. The cohomology of $\G\b G/K$ with complex coefficients, can be written as a direct sum of relative Lie algebra cohomologies of $(\f{g},K)$ with coefficients in $A_\f{q}$, where $\f{q}$ are $\theta$-stable parabolic subalgebras of $\f{g}_0$ and $A_\f{q}$ are obtained by cohomological induction on the trivial $\f{q}$-module. This is called the \emph{Matsushima decomposition}.
%the Harish Chandra modules of irreducible unitary representations which occur in the discrete decomposition of $L^2(\G\b G)$. 
%The summands are perpendicular with respect to the inner product on $H^*(\G\b G/K)$ induced from the Riemannian metric on $\G\b G/K$.
%The $(\f{g},K)$-modules for which the relative Lie algebra cohomology is non-zero are of the form $A_\f{q}$,  The cohomology groups $H^*(\f{g},K;A_\f{q})$ are well known.
Our aim is to construct concrete cohomology classes of $\G\b G/K$ and determine which of their components, with respect to the Matsushima decomposition, are non-zero. A non-zero component implies that the corresponding irreducible unitary representation occur in $L^2(\G\b G)$.  

The cohomology classes that will be considered are Poincar\'e duals of totally geodesic submanifolds. We call these \emph{geometric classes}. In literature these submanifolds are variously called \emph{geometric cycles, special cycles} or \emph{generalized modular symbols}.   Millson and Raghunathan \cite{mr} proved that if a pair of complementary dimensional totally geodesic submanifolds have all the intersection numbers  positive, then (going to a finite cover if necessary) the corresponding geometric classes have non-zero components other than the one corresponding to the trivial representation. Unfortunately, to our knowledge, there is no other result about non-vanishing of components of a geometric class. But since the number of non-zero components can only be finite, we go the roundabout way of ascertaining the representations for which the corresponding components of a geometric class are zero. For this we restrict to the case of compact Hermitian locally symmetric spaces. Here we exploit the additional Hodge bigrading on the cohomology. If the submanifold is complex analytic then the corresponding geometric class is of $(p,p)$-type. On the other hand, if $(p,q)$ denotes the Hodge types of classes in $H^*(\f{g},K;A_\f{q})$, then $q-p$ is constant. Thus a complex analytic geometric class has no $A_\f{q}$-component for which the above constant in non-zero. This is the first (and elementary) vanishing result on components of a geometric class that we will use. It was already noted and exploited in \cite{monsan}. A non-trivial vanishing result was obtained by Kobayashi and Oda in \cite{koboda}. We generalize this to arrive at the second non-vanishing result (Corollary \ref{main2} in \S\ref{nontrivial}) that we will use:
\begin{theorem}
Assume $G$ is simple. %Let $K$ be a maximal compact subgroup and $\G$ be a uniform lattice in $G$.
Let $G'\subset G$ and $K':=G'\cap K, \G':=G'\cap \G$ be such that $\G'\b G'/K'$ is a compact totally geodesic submanifold of $\G\b G/K$. Let $\f{g}'$ be the complexified Lie algebra of $G'$. If $A_\f{q}$ is discretely decomposable as a $(\f{g}',K')$ module, then the Poincar\'e dual of $\G'\b G'/K'$ has no $A_\f{q}$-component.
\end{theorem}
From a computational point of view, this result owes its significance to a simple criterion, given by Kobayashi in \cite{kob94},\cite{kob97}, for discrete decomposability of $A_\f{q}$, when regarded as module over a reductive subalgebra. In fact, using this criterion, Kobayashi and Oshima list in \cite{kobosh} all symmetric pairs $(\f{g}_0,\f{g}'_0)$ and the modules $A_\f{q}$, such that $A_\f{q}$ is discretely decomposable as a $(\f{g}',K')$-module.
%We reproduce here the part of the list for which $\f{g}'$ is associated to a Hermitian symmetric space (so that the corresponding submanifold is complex analytic).
Our strategy is to list all the cohomologies with $A_\f{q}$ coefficients that are non-zero in the dimension of a geometric class and eliminate as many as possible based on the vanishing results. In some cases, this method of elimination leaves out just one $A_\f{q}$ coefficient. This allows us to conclude that our geometric class has that particular component. Thus vanishing results allow us to obtain non-vanishing results in certain cases.  Following is our main theorem.
\begin{theorem}\label{main}
Let $G=\emph{SU}(p,q), p\le q$.
%or $\emph{SO}_0(2,4)$.
Consider the $(\f{g},K)$-modules $A_{\f{q}_\lambda}$ with parameters $\lambda$ and conditions on $p,q$ as given below. \emph{(}Refer to \S5 to see how these parameters are defined.\emph{)}

\vspace{.1cm}

%\begin{enumerate}
    \noindent\emph{(1)} $\lambda=\epsilon_1-\epsilon_p$ and $5\le p\le q\le 2p-2$.
    
    \vspace{.1cm}
    
    \noindent\emph{(2)} $\lambda=\epsilon_{p+1}-\epsilon_{p+q}$ and  $5\le p\le q, p\ne q-1$.
    
    \vspace{.1cm}
    
    %\item $\f{g}_0=\f{so}(2,4)$ and $\lambda=\epsilon_2+\epsilon_3$ or $\epsilon_2+\epsilon_3$.
%\end{enumerate}
%If $p<q$, there are unique representations $A_\f{q}$, such that $R^+(\f{q})=R^-(\f{q})=p$ and $R^+(\f{q})=R^-(\f{q})=q$. If $p=q$ there are two representations $A_\f{q}$ such that $R^+(\f{q})=R^-(\f{q})=p$.

%Consider the $\theta$-stable parabolic subalgebras $\f{q}_\lambda$ parametrized by the following dominant elements  $\Phi$ in   $\lambda$ are as follows :
% \begin{enumerate}
%     \item In the second case, there are two representations $A_\f{q}$ such that $R^+(\f{q})=R^-(\f{q})=2$. 
\noindent For each of these $A_{\f{q}_\lambda}$, there exists a
%commensurability class $\mathcal{L}$ of uniform arithmetic lattices, such that for each
torsion-free uniform lattice $\G$ in $G$ and a geometric class in $H^*(\G\b G/K;\c)$ having an $A_{\f{q}_\lambda}$-component. Hence all these $A_{\f{q}_\lambda}$ occur with non-zero multiplicity in $L^2(\Gamma\b G)$.
\end{theorem}

Now we give a section wise description. In \S\ref{md} we give a description of Matsushima's decomposition, the groups $H^*(\f{g},K;A_\f{q})$ and the Hodge bigrading on it when $\f{g}_0$ is associated to a Hermitian symmetric space. In \S\ref{cons} we describe a construction which ensures that given an involution $\sigma$ of $G$, there exists a maximal compact subalgebra $K$ and a uniform torsion-free lattice $\G$ of $G$, such that $\sigma$ induces an involution of $\G\b G/K$ and the Poincar\'e dual of the fixed point submanifold is not $G$-invariant. In \S\ref{nontrivial} we extend Kobayashi and Oda's vanishing result. In \S\ref{proof} we prove Theorem \ref{main}. 

\section{Matsushima decomposition of $H^*(\G\b G/K;\c)$}\label{md}

%Let $G$ be a linear semisimple Lie group with no compact normal subgroups. Let $K$ be a maximal compact subgroup of $G$. Let $\G$ be a uniform lattice in $G$.
In this section we wish to describe the de Rham cohomology of $X_\G:=\G\b G/K$ in terms of cohomology of $(\f{g},K)$-modules. Let us first do this for the universal cover $X:=G/K$.  At each point of a manifold, a complex valued differential $m$ form takes an element of the $m^\text{th}$ exterior product of the complexified tangent space and returns a complex number, and it does so in a smooth manner. Since $X$ is a homogeneous manifold the tangent spaces and their exterior products at each point can be identified with those at $eK$ via the action of $G$.
%Let $\f{g}_0$ and $\f{k}_0$ be the Lie algebras of $G$ and $K$ respectively. 
Let $\f{g}_0=\f{k}_0\oplus\f{p}_0$ be the Cartan decomposition with respect to a Cartan involution $\theta$. The tangent space of $X$ at $eK$ can be identified with $\f{p}_0$. Thus a complex valued differential $m$ form on $X$ can be thought of an element in $\text{Hom}_K(\wedge^m\f{p},C^\infty(G))$, where $C^\infty(G)$ is the space of complex valued smooth functions on $G$ equipped with the right translation action of $G$. The $K$ equivariance takes care of the ambiguity arising from the non-uniqueness of elements of $G$ that can be chosen for identifying the complexified tangent space at a point with $\f{p}$.  Since $X$ is a cover of $X_\G$, the differential forms of $X_\G$ are just $\G$-invariant differential forms of $X$. Thus a complex valued differential $m$ form of $X_\G$ can be identified with an element in $\text{Hom}_K(\wedge^m\f{p},C^\infty(\G\b G))$. Since $\wedge^m\f{p}$ is finite dimensional, its image under a $K$-homomorphism must land in the $K$-finite part $C^\infty(\G\b G)_K$ of $C^\infty(\G\b G)$. For any $(\f{g},K)$-module $V$, let $C^*(\f{g},K;V)$ denote the cochain complex of the relative Lie algebra cohomology of $(\f{g},K)$ with coefficients in $V$, where the individual cochain groups are defined as $C^m(\f{g},K;V)=\text{Hom}_K(\wedge^m\f{p},V)$. To summarise the above discussion, we have an isomorphism of chain complexes of the de Rham complex $\Omega^*(X_\G;\c)$ and $C^*(\f{g},K;C^\infty(\G\b G)_K)$. See \cite[Chapter VII, Proposition 2.5]{borel-wallach}.

If $M$ is a compact oriented Riemannian manifold then there is a inner product on  $\Omega^*(M;\c)$ given by $\langle\omega,\eta\rangle=\int_M\omega\wedge*\eta,$ where $*$ is the Hodge star operator. On the other hand any relative Lie algebra cohomology with coefficients in a unitary $(\f{g},K)$-module $V$ has an inner product on it defined as follows. The Killing form on $\f{g}_0$ restricts to an inner product on $\f{p}_0$. The representation on $V$ being unitary, it comes equipped with an inner product. So via the isomorphism  $\text{Hom}_K(\wedge^m\f{p},V)\cong(\wedge^m\f{p}^*\otimes V)^K$ we get an inner product on $\text{Hom}_K(\wedge^*\f{p},V)$. In particular if $V=C^\infty(\G\b G)_K$, then inner product on $V$ comes from the inclusion $C^\infty(\G\b G)_K\subset L^2(\G\b G)$. Thus both the vector spaces $\Omega^*(X_\G;\c)$ and $C^*(\f{g},K;C^\infty(\G\b G)_K)$ are equipped with inner products. The isomorphism of chain complexes between them is also an isometry. 

As in case of de Rham complexes of compact orientable Riemannian manifolds, we may define the Laplace operator $\Delta$ on $C^*(\f{g},K;V)$, for any $(\f{g},K)$-module $V$, as $\Delta=d\partial+\partial d$, where $\partial$ is adjoint of $d$ with respect to the inner product. The action of the Laplace operator is equal to that of the Casimir element $c\in U(\f{g})$. If $V$ is an irreducible $(\f{g},K)$-module then $H^*(\f{g},K;V)\ne 0$ only if the action of Casimir element is trivial. In this case $H^*(\f{g},K;V)=C^*(\f{g},K;V)=\text{Hom}_K(\wedge^*\f{p},V)$. See \cite[Chapter II, Proposition 3.1]{borel-wallach}.

By Hodge theory $H^*(X_\G,\c)=\mathcal{H}^*(X_\G)$, where $\mathcal{H}^*(X_\G)$ is the space of complex valued harmonic forms. On the other hand $H^*(X_\G,\c)=H^*(\f{g},K;$ $C^\infty(\G\b G)_K)$. The $(\f{g},K)$-module $C^\infty(\G\b G)_K$ is the Harish Chandra module of $L^2(\G\b G)$. By \cite[Chapter 1, \S2]{gp}, $L^2(\G\b G)$ is  discretely decomposable as a $G$-module: $L^2(\G\b G)\cong\hat{\bigoplus}_{\pi\in\widehat{G}}m(\pi,\G)H_\pi$, where the multiplicities $m(\pi,\G)$ are finite. From all this information one can deduce the \emph{Matsushima decomposition}:  

$$\mathcal{H}^*(X_\G)\cong\bigoplus_{\pi\in\widehat{G}_0}m(\pi,\G)\text{Hom}_K(\wedge^*\f{p},H_{\pi,K}),$$
where $\widehat{G}_0:=\{\pi\in\widehat{G}:\pi(c)=0\}$, where $c$ is the Casimir element and, by abuse of notation, $\pi$ is also the corresponding action of $U(\f{g})$ on the Harish-Chandra module $H_{\pi,K}$. See \cite[Chapter VII, Corollary 3.3]{borel-wallach}.

\begin{definition}\cite[(2.3.2)]{koboda}
For $\pi\in\hat{G}_0$, we define the \emph{$\pi$-component} of $\mathcal{H}^*(X_\G)\subset\text{Hom}_K(\wedge^*\f{p},$ $C^\infty(\G\b G)_K)$ to be consisting of those homomorphisms whose images are contained in the $\pi$-isotypical component of $C^\infty(\G\b G)_K$. That is, $\omega\in\pi\text{-component}$ if and only if it is a sum of homomorphisms that factor through $H_{\pi,K}$:
\begin{center}
\begin{tikzpicture}
\node at (0,1.5) {$\wedge^*\f{p}$};
\node at (2,1.5) {$H_{\pi,K}$};
\node at (4.5,1.5) {$C^\infty(\G\b G)_K$,};
\draw [->] (.45,1.5)--(1.5,1.5);
\draw [->] (2.5,1.5)--(3.3,1.5);
\node [above] at (1.15,1.5) {$\psi$};
\node [above] at (2.85,1.5) {$\phi$};
\end{tikzpicture}
\end{center}
where $\psi$ is a $K$-map and $\phi$ is a $(\f{g},K)$-map.
%The corresponding cohomology classes constitute the \emph{$\pi$-component} of the cohomology groups. 
\end{definition}

%We summarize the discussion right before the definition as a lemma.

%Now we describe the irreducible unitary $(\f{g},K)$-modules for which the relative Lie algbera cohomology is non-zero. 
%Let $\theta$ be the Cartan involution that pointwise fixes the Lie algebra $\f{k}_0$ of $K$.

The relative Lie algebra cohomology $H^*(\f{g},K;H_{\pi,K})$ is non-zero if and only if $H_{\pi,K}$ is of the form $A_\f{q}$, where $\f{q}$ is a $\theta$-stable parabolic subalgebra of $\f{g}_0$ and $A_\f{q}$ is the $(\f{g},K)$-module obtained by cohomological induction on the trivial one dimensional $(\f{q},L\cap K)$-module. (The subgroup $L$ is defined below.)  See \cite[Theorem 4.1]{vz}. Let $\f{q}=\f{l}+\f{u}$, where $\f{l}$ is the Levi part and $\f{u}$ is the nilpotent radical. By definition of $\theta$-stable parabolic subalgebra, $\f{l}=\f{q}\cap\bar{\f{q}}$ and hence $\f{l}_0:=\f{l}\cap\f{g}_0$ is a real form of $\f{l}$. Let $L$ be the connected subgroup of $G$ with Lie algebra $\f{l}_0$. By \cite[Theorem 3.3]{vz}, we have
$$H^k(\f{g},K;A_\f{q})=H^{k-R(\f{q})}(\f{l},L\cap K;\c)=\text{Hom}_{L\cap K}(\wedge^{k-R(\f{q})}(\f{l}\cap\f{p}),\c),$$
where $R(\f{q})=\dim(\f{u}\cap\f{p})$.
%Let $\f{l}_u:=(\f{l}_0\cap\f{k}_0)+i(\f{l}_0\cap\f{p}_0)$ be the compact dual of $\f{l}_0$. Let $L_u$ be the connected subgroup of the complexification $G^\c$ of $G$, whose Lie algebra is $\f{l_u}$. Note that $\text{Hom}_{L\cap K}(\wedge^*(\f{l}\cap\f{p}),\c)$ is isomorphic to the space $\text{Hom}_{L_u\cap K}(\wedge^*(i(\f{l}\cap\f{p})),\c)$ of invariant forms on the compact homogeneous space $L_u/L_u\cap K$. Hence $H^*(\f{l},L\cap K;\c)\cong H^*(L_u/(L_u\cap K);\c)$.
The value of the shift $R(\f{q})$ is $0$ if and only if $\f{q}=\f{g}$ and in this case $A_\f{g}=\c$, where the action of $(\f{g},K)$ on $\c$ is trivial. Let us denote this trivial one dimensional representation by $(\mathbf{1},\c)$. Then $m(\textbf{1},\G)=1$ for all uniform lattice $\G$. In fact this is the submodule of $C^\infty(\G\b G)_K$ consisting of constant functions. So the corresponding forms are the $G$-invariant ones.

\begin{lemma}\label{top}\emph{\cite[Lemma 2.4]{koboda}}
The only $\pi$-component that contributes to the top cohomology group is the $\mathbf{1}$-component.
\end{lemma}

%A symmetric space $G/K$ is Hermitian if and only if $K$ has a continuous centre. That is the centre $Z(\f{k})$ of $\f{k}$ is non-zero. We have the following inclusions: $Z(\f{k}_0)$ 
Now suppose $X$ is Hermitian symmetric. Then $X_\G$ is a K\"ahler manifold and hence its cohomology has a Hodge bigrading. In this case, relative Lie algebra cohomology of $(\f{g},K)$ with coefficient in any module also has a Hodge bigrading. See \cite[Chapter II, Corollary 4.5]{borel-wallach}. Taking this into account, Matsushima's decompostion can be written as
$$H^{p,q}(X_\G)\cong\bigoplus_{\pi\in\hat{G}_0}m(\pi,\G)H^{p,q}(\f{g},K;H_{\pi,K}).$$
The complex structure $J$ on $\f{p}_0$ induces one on $\f{l}_0\cap\f{p}_0$. Hence the groups $H^*(\f{l},L\cap K;\c)$ also have Hodge decomposition. By \cite[Proposition 6.19]{vz}, the Hodge bigradings of $H^*(\f{g},K;A_\f{q})$ and $H^*(\f{l},L\cap K;\c)$  are related as 
$$H^{p,q}(\f{g},K;A_\f{q})\cong H^{p-R^+(\f{q}),q-R^-(\f{q})}(\f{l},L\cap K;\c),$$
where $R^+(\f{q})=\dim(\f{u}\cap\f{p}^+)$ and $R^-(\f{q})=\dim(\f{u}\cap\f{p}^-)$. Here $\f{p}^+$ and $\f{p}^-$ are the $+i$ and $-i$ eigenspaces of $J$ in $\f{p}$. In this case $L_u/(L_u\cap K)$ is a compact Hermitian symmetric space and hence all its cohomology classes are of Hodge type $(p,p)$. Thus $H^{p,q}(\f{g},K;A_\f{q})\ne 0$ if and only if $p-q=R^+(\f{q})-R^-(\f{q})$. We record the particular case of $p=q$ as a lemma.

\begin{lemma}\label{hol}
The vector space $H^{p,p}(\f{g},K;A_\f{q})\ne 0$ only if $A_\f{q}$ satisfies $R^+(\f{q})=R^-(\f{q})$.
\end{lemma}

% The corresponding spaces 

\section{Construction of geometric classes}\label{cons}

The construction of non-$G$-invariant geometric class in $H^*(\G\b G/K;\c )$, for some torsion-free uniform lattice $\G$, has reduced to fixing an involution of $G$, thanks to the results of several people. The first crucial step in this direction was taken by Millson and Raghunathan in \cite{mr}. They show that if one can  construct pairs of complementary dimensional totally geodesic submanifolds of $\G\b G/K$ such that they intersect in a finite set and the intersection number at each point of this finite set is positive, then going to a finite cover if necessary, the Poincar\'e duals of these submanifolds are not $G$-invariant. See \cite[Theorem 2.1]{mr}. To construct complementary dimensional submanifolds they start with an involution $\sigma$ of $G$ which commutes with a Cartan involution $\theta$ and assume that there exists a uniform torsion-free arithmetic lattice which is invariant under both $\sigma$ and $\theta$. That, given a $\sigma$, such a Cartan involution and a cocompact lattice will always exist was later proved by Raghunathan in \cite{msr} using an earlier construction of Borel \cite{borel}. Since $\sigma\theta=\theta\sigma$, the fixed point subgroup $K$ of $\theta$ in $G$ remains invariant under $\sigma$. Let $A(\sigma)$ denote the fixed point subset of a set $A$ under the action of $\sigma$.  Then $\G(\sigma)\b G(\sigma)/K(\sigma)$ and $\G(\theta\sigma)\b G(\theta\sigma)/K(\theta\sigma)$ are their candidates for submanifolds whose intersection numbers are all positive.  Note that $ G(\sigma)/K(\sigma)$ and $G(\theta\sigma)/K(\theta\sigma)$ are complementary dimensional and intersect at exactly one point. Then a criterion, using the first Galois cohomology of algebraic groups, for $\G(\sigma)\b G(\sigma)/K(\sigma)$ and $\G(\theta\sigma)\b G(\theta\sigma)/K(\theta\sigma)$ to intersect in a finite set with all positive intersection numbers, is given. Using  this criterion they obtain non-vanishing results in case of $G=\text{Sp}(p,q),\text{SU}(p,q)$ and $\text{SO}(p,q)$. Rohlfs and Schwermer in \cite{rs} remove the necessity of this criterion by showing that, under a mild orientability condition, the intersection numbers will always be positive if we go to a finite cover. Their result is proved in a greater generality, but we will not need it here. We summarize this discussion in the form of a theorem.

%Then there exists a $\q$-structure of $G$ and a Cartan involution $\theta$ such $\sigma\theta=\theta\sigma$ and both $\sigma$ and $\theta$ are $\q$-morphisms. There exists an arithmetic lattice $\G$ such that $\sigma$ and $\sigma\theta$ induces involutive isometries of $\G\b G/K$. Under a mild orientability condition the fixed point submanifolds are geometric cycles. 

% \textbf{Some notations:} Let $\sigma,\theta$ and $\G$ be as before. The fixed point set under $\sigma$ of any subset $A$ of $G$ will be denoted by $A'$. For example the fixed point set of $G$ is $G'$. Then $Y:=\G'\b G'/K'$ is a geometric cycle. Let $\iota:Y\hookrightarrow X$ be the inclusion map. It induces a map $\iota_*:H_*(Y)\to H_*(X)$ in homology. Let $[Y]$ be the fundamental homology class of $Y$.  Let $\mathcal{P}(Y)$ denote the Poincar\'e dual of $\iota_*([Y])$.

\begin{theorem}\label{con}~

\vspace{.1cm}

\noindent\emph{(1) (Borel \cite{borel}, Raghunathan \cite{msr})}
Let $G$ be a connected linear semisimple Lie group. Let $\sigma$ be any involution of $G$. Then there exists a global Cartan involution $\theta$ of $G$ such that $\sigma\theta=\theta\sigma$ and a cocompact arithmetic lattice $\Lambda$ in $G$ which is invariant under $\sigma$ and $\theta$.

\vspace{.1cm}

\noindent\emph{(2) (Rohlfs and Schwermer \cite{rs})} With $\sigma,\theta$ as above and $G(\theta)=K$, $\sigma$ induces isometries of $G/K$ and $\G\b G/K$ which we again denote by $\sigma$. Under an orientability condition, which is always satisfied if $G/K$ is Hermitian and $\sigma$ is holomorphic, there exists a finite index torsion-free subgroup $\G$ of $\Lambda$, which  is again invariant under $\sigma$ and $\theta$, such that $\G(\sigma)\b G(\sigma)/K(\sigma)$ and $\G(\theta\sigma)\b G(\theta\sigma)/K(\theta\sigma)$ have all their intersection numbers positive.

\vspace{.1cm}

\noindent\emph{(3) (Millson and Raghunathan \cite{mr})} Replacing $\G$ by a finite index subgroup if necessary, the Poincar\'e duals of $\G(\sigma)\b G(\sigma)/K(\sigma)$ and $\G(\theta\sigma)\b G(\theta\sigma)/K(\theta\sigma)$ are not $G$-invariant.

%To summarize, if $\G\b G/K$ is a compact hermitian symmetric space then, given an involution $\sigma$ of $G$, which keeps the centre of $K$ fixed, one can construct a pair of complementary dimensional non $G$-invariant geometric classes of $(p,p)$-type.
%Moreover the fixed point set $\G(\sigma)$ is cocompact in $G(\sigma)$.  $\q$-structure of $\f{g}_0$ and a Cartan involution $\theta$, such that both $\sigma$ and $\theta$ are $\q$-rational and $\sigma\theta=\theta\sigma$. Then there exists  
\end{theorem}

% \begin{remark}
% Borel and Raghunathan use an abstract construction of $F$-structure of any non-compact semisimple Lie algebra 
% The generality of the above result means that when we depend solely on it for construction of a non-$G$-invariant geometric class, we cannot say for which particular lattice (up to commensurability) it exists. Removal of this drawback will require launching into a case by case description of totally real $F$-structures of $G$. Many of the paper dealing with construction of geometric classes do this. For example, see \cite{sw1},\cite{sw2},\cite{schimpf} and \cite{monsanun}. In this paper we steer clear of such considerations since our main objective is to extend the vanishing result of Kobayashi and Oda and show how it can be applied to obtain non-vanishing results.   
% \end{remark}

%Let $\G_1\subset\G\subset\G_2$ be torsion-free uniform lattices in $G$. The induced maps
% \begin{center}
% \begin{tikzpicture}
% \node at (0,1.5) {$\G_1\b G/K$};
% \node at (2,1.5) {$\G\b G/K$};
% \node at (4.5,1.5) {$\G_2\b G/K$,};
% \draw [->] (.5,1.5)--(1.8,1.5);
% \draw [->] (2.2,1.5)--(3.3,1.5);
% \node [above] at (1.15,1.5) {$\pi_1$};
% \node [above] at (2.85,1.5) {$\pi_2$};
% \end{tikzpicture}
% \end{center}
% are covering maps. Let $Y$ be a totally geodesic submanifold of $\G\b G/K$. Then $\pi_1^{-1}(Y)$ and $\pi_2(Y)$ are totally geodesic submanifolds in $\G_1\b G/K$ and $\G_2\b G/K$, respectively. 

As indicated before we will concentrate on the case where $G/K$ is Hermitian. Assuming $G$ is simple, this happens if and only if $K$ has a non-discrete centre. In fact in this case $Z(K)\cong S^1$ and the complex structure is given by $\text{Ad}(j)|_{\f{p}_0}$, where $j\in Z(K)$ such that order of $\text{Ad}(j)$ is $4$. See \cite[Chapter VII, \S6]{hel}. Thus an involution $\sigma$ of $G$ is holomorphic if and only if $\sigma$ point-wise fixes $Z(K)$. In this case, both $\G(\sigma)\b G(\sigma)/K(\sigma)$ and $\G(\theta\sigma)\b G(\theta\sigma)/K(\theta\sigma)$ are complex analytic submanifolds. Hence their Poincar\'e duals are of $(p,p)$-type. This observation along with Lemma \ref{hol} yields the following vanishing result. 

\begin{lemma}\label{comana}
Let $G/K$ be a Hermitian symmetric space and $\sigma$ be an involution of $G$ that keeps $Z(K)$ fixed (pointwise). Let $\omega$ and $\omega'$ be the pair of geometric classes associated to $\sigma$, as explained in Theorem \ref{con}. If $A_\f{q}$ satisfies $R^+(\f{q})\ne R^-(\f{q})$, then $\omega$ and $\omega'$ have no $A_\f{q}$-component.  
\end{lemma}

\begin{remark}
Even this simple vanishing result yields non-vanishing ones in quite a few cases. See \cite[Theorem 1.1]{monsan}.
\end{remark}

\section{A non-trivial vanishing result}\label{nontrivial}

Now we come to the most important vanishing result that we will use. Let $\sigma$ be an involutive automorphism of $G$ which commutes with a Cartan involution $\theta$. Since we will be dealing with only one involution here, we change the notation, for fixed point subset of a set $A$ under $\sigma$, from $A(\sigma)$ to $A'$.
%The fixed point set under $\sigma$ of any subset $A$ of $G$ will be denoted by $A'$. For example the fixed point set of $G$ is $G'$. 
Let $\G$ be a $\sigma$-invariant torsion-free uniform lattice in $G$ such that  $\G'$ is a lattice in $G'$. (We will always assume that $G'$ is non-compact.) Then $Y_\G:=\G'\b G'/K'$ is a totally geodesic submanifold of $X_\G:=\G\b G/K$. Let $\iota:Y_\G\hookrightarrow X_\G$ be the inclusion map.
%It induces a map $\iota_*:H_*(Y_\G)\to H_*(X_\G)$ in homology. Let $[Y_\G]$ be the fundamental homology class of $Y_\G$.
Let $\mathcal{P}(Y_\G)$ denote the harmonic form that represents the Poincar\'e dual of $Y_\G$. Given an irreducible unitary representation $(\pi,V_\pi)$ of $(\f{g},K)$, on which the Casimir operator acts trivially,  Kobayashi and Oda in \cite[Theorem 2.8]{koboda} gave a criterion for $\mathcal{P}(Y_\G)$ to not have a $\pi$-component. One of the conditions is \emph{discrete decomposability} which we now define.

\begin{definition}
Let $V$ be a unitary $(\f{g},K)$-module. We say that $V$ is \emph{discretely decomposable} if it can be written as a direct sum of irreducible $(\f{g},K)$-modules.
\end{definition}

\begin{remark}
In \cite[\S1]{kob97}, Kobayashi gives a more general definition of discrete decomposability of any $(\f{g},K)$-module and proves that his definition is equivalent to the one above in the unitary case. We work with the above definition since it is more intuitive.
\end{remark}

We update Kobayashi and Oda's result by removing one of the conditions and present it as two separate statements in the form of a Theorem and a Corollary. This is done to bring out the key point of their result. The proof follows their approach.

\begin{theorem}\label{main1}
With notations as above, if $\emph{Hom}_{(\f{g}',K')}(V_\pi,\c)=0$, where the action of $(\f{g}',K')$ on $\c$ is trivial, then $\mathcal{P}(Y_\G)$ does not have a $\pi$-component.
\end{theorem}

\begin{cor}\label{main2}
Let $G$ be simple. If $\mathbf{1}\ne V_\pi$ is discretely decomposable as a $(\f{g}',K')$-module, then $\mathcal{P}(Y_\G)$ does not have a $\pi$-component.
\end{cor}

\begin{proof}\textit{of Theorem.} 
Let $\langle~,~\rangle$ be the inner product on $\mathcal{H}^*(X_\G)$. Since the components are all mutually orthogonal, it is enough to show that $\langle\mathcal{P}(Y_\G),\omega\rangle=0$, for all $\omega\in\pi$-component. By definition of Poincar\'e dual, $\langle\mathcal{P}(Y_\G),\omega\rangle=\int_{X_\G}\mathcal{P}(Y_\G)\wedge*\omega=\int_{Y_\G}\iota^*(*\omega)$, where $\iota^*$ is the pull back induced by the inclusion $\iota$. Since the Hodge $*$ operator takes the $\pi$-component to itself, it is enough to prove that $\int_{Y_\G}\iota^*(\omega)=0$ for all $\omega\in\pi$-component.

%If $M$ is a compact oriented Riemanninan manifold then there is an inner product on the space of  Let $M$ be a orientable $n$-dimensional manifold. Then for we have a pairing of  
% \begin{align*}

%     H^k(M)\times H^{n-k}(M)&\to\c\\
%     (\omega,\eta)\times\int_M\omega\wedge\eta 
% \end{align*}

%In particular we have an inner product $\langle~,~\rangle$ on $H^*(\G\b G/K)$. With respect to this inner product the Matsushima decomposition is orthogonal. Thus the cohomology class $\mathcal{P}(Y)$ is not in the $\pi$-component if and only if $\langle\mathcal{P}(Y),\omega\rangle=0$ for all $\omega\in\pi$-component.By definition of Poincar\'e duality, $\langle\mathcal{P}(Y),\omega\rangle=$
%One observes that $\langle\omega,\iota^*([Y])\rangle=\langle\iota^*(\omega),[Y]\rangle$, where $\iota^*:\Omega^*(X)\to \Omega^*(Y)$ is induced by the inclusion $\iota:Y\hookrightarrow X$. This easy observation is crucial in that it transfers the problem from cohomology of $X$ to that of $Y$.

Recall from \S\ref{md}, $\Omega^*(X_\G;\c)\cong\text{Hom}_K(\wedge^*\f{p},C^\infty(\G\b G)_K)$ and similarly for $\Omega^*(Y_\G;\c)$. Let us understand the induced map
$\iota^*:\text{Hom}_K(\wedge^*\f{p},C^\infty(\G\b G)_K)\to\text{Hom}_{K'}(\wedge^*\f{p}',C^\infty(\G'\b G')_{K'})$. Let $\omega:\wedge^*\f{p}\to C^\infty(\G\b G)_K$ be a $K$-equivariant homomorphism. Then its image under $\iota^*$ is given by the lower horizontal map in the figure below which makes the diagram commutes. 
\begin{center}
\begin{tikzpicture}[node distance=3cm,auto]
\node (A) {$\wedge^*\f{p}$};
\node (B) [right of=A] {$C^\infty(\G\b G)_K$};
\node (C) [below=1cm of A] {$\wedge^*\f{p}'$};
\node (D) [below=1cm of B] {$C^\infty(\G'\b G')_{K'}$};
\draw [->] (C) to node {$j$} (A);
\draw [->] (A) to node {$\omega$} (B);
\draw [->] (B) to node {$i^*$} (D);
\draw [->, dashed] (C) to node {}(D);
\end{tikzpicture}
\end{center}
Here $j:\wedge^*\f{p}'\to\wedge^*\f{p}$ is induced by the natural inclusion $\f{p}'\hookrightarrow\f{p}$ and $i^*:C^\infty(\G\b G)_K\to C^\infty(\G'\b G')_{K'}$ is the restriction map. Thus $\iota^*(\omega)=i^*\circ\omega\circ j$. Since $\omega$ is in the $\pi$-component, we may assume that it is of the form $\phi\circ\psi$, where $\phi:\wedge^*\f{p}\to V_\pi$ is a $K$-equivariant map and $\psi:V_\pi\to C^\infty(\G\b G)_K$ is  a $(\f{g},K)$-equivariant map.
\begin{center}
\begin{tikzpicture}
\node at (0,1.5) {$\wedge^*\f{p}$};
\node at (2,1.5) {$V_\pi$};
\node at (5,1.5) {$C^\infty(\G\b G)_K$};
\node at (0,0) {$\wedge^*\f{p}'$};
\node at (5,0) {$C^\infty(\G'\b G')_{K'}$};
\node at (5,-1.5) {$\c$};
\draw [->] (.5,1.5)--(1.7,1.5);
\draw [->, dashed] (2,1.2)--(4.8,-1.5);
\draw [->] (2.3,1.5)--(3.9,1.5);
\draw [->] (0,.25)--(0,1.25);
\draw [->] (5,-.25)--(5,-1.25);
\draw [<-] (5,.25)--(5,1.25);
\node [left] at (0,.75) {$j$};
\node [right] at (5,.75) {$i^*$};
\node [above] at (1.15,1.5) {$\psi$};
\node [above] at (3.15,1.5) {$\phi$};
\node [right] at (5,-.75) {$\text{pr}_\mathbf{1}$};
\end{tikzpicture}
\end{center}
If $\iota^*(\omega)$ is not a top cohomology class of $Y$, then $\int_{Y_\G}\iota^*(\omega)=0$. So let us assume that $\iota^*(\omega)$ is a top cohomology class. Now by Lemma \ref{top} a top dimensional form is not exact if and only if it has a non-zero $G'$-invariant component. Thus $\iota^*(\omega)\ne 0$ means that its image has a component in the trivial one dimensional sub-representation of $C^\infty(\G'\b G')_{K'}$, consisting of constant functions. Let $\text{pr}_\mathbf{1}:C^\infty(\G'\b G')_{K'}\to\c$ be the projection map to the subspace of constant functions. Then the composition $\text{pr}_\mathbf{1}\circ i^*\circ\phi\circ\psi\circ j$ is a non-trivial homomorphism. In particular $\text{pr}_\mathbf{1}\circ i^*\circ\phi:V_\pi\to\c$ is a non-trivial $(\f{g}',K')$-homomorphism. But this contradicts our hypothesis. Hence $\mathcal{P}(Y_\G)$ does not have a $\pi$-component.
\end{proof}

\begin{remark}
It was brought to our notice that in his 2009 preprint \cite{kob09} Kobayashi had proved that if $\text{Hom}_{\f{g},K}(V_\pi,C^\infty(G'\b G))=0$ then $\mathcal{P}(Y_\G)$ has no $\pi$-component. This is same as Theorem \ref{main1}, since by Frobenius reciprocity $\text{Hom}_{\f{g},K}(V_\pi,C^\infty(G'\b G))=\text{Hom}_{(\f{g}',K')}(V_\pi,\c)$. But the proof is different. Kobayashi constructs an intertwining operator $R:C^\infty(\G\b G)\to C^\infty(G'\b G)$ by averaging over $\G'\b G'$. Then he shows that $\int_{Y_\G}\iota^*(\omega)=\int_{Y_\G}\iota^*(\phi\circ\psi)=\text{ev}_e\circ R\circ\phi\circ\psi\circ j(\text{vol})$, where $\text{ev}_e$ is the evaluation map at $e$ and $\text{vol}$ is the oriented norm one element of $\wedge^{\dim\f{p}}\f{p}$. Since $R\circ\phi\in\text{Hom}_{\f{g},K}(V_\pi,C^\infty(G'\b G))=0$, we have $\int_{Y_\G}\iota^*(\omega)=0$.
\end{remark}

\begin{proof}\textit{of Corollary.}
We will show that if $V_\pi$ is discretely decomposable as a $(\f{g}',K')$-module, then it satisfies the condition of Theorem \ref{main1}. We have
$V_\pi\cong\bigoplus_\tau m_\pi(\tau)V_\tau,$
where $\tau$ runs over irreducible unitary $(\f{g}',K')$-modules and $m_\pi(\tau)$ are the multiplicities. By \cite[Proposition 1.6]{kob97},
$\bar{V}_\pi\cong\hat{\bigoplus}_\tau m_\pi(\tau)\bar{V}_\tau,$ where  $\bar{V}_\tau$ and $\bar{V}_\pi$ are  the representation spaces of $G'$ obtained by taking completion of $V_\tau$ and $V_\pi$, respectively. Because of discrete decomposability we have $m_\pi(\mathbf{1})=\dim\text{Hom}_{\f{g}',K'}(\c,V_\pi)=\dim\text{Hom}_{\f{g}',K'}(V_\pi,\c)$. Let us assume that $\dim\text{Hom}_{\f{g}',K'}(V_\pi,\c)\ne 0$.  Then $(\mathbf{1},\bar{\c})=(\mathbf{1},\c)$ occurs as a subrepresentation of $(\pi,\bar{V}_\pi)$. Now we state a result of Moore \cite[Theorem 1]{moore}. Suppose that $G'$ is a non-compact subgroup of a simple Lie group $G$. Then $G'$ has the property that if the restriction to $G'$, of any unitary representation of $G$, has a fixed vector then this vector is also fixed under the $G$ action. In our context this means that $\bar{V}_\pi$ has a one dimensional $G$-invariant subspace. But it is given that $V_\pi$, and hence $\bar{V}_\pi$, is irreducible. Thus we arrive at a contradiction. This finishes the proof.
\end{proof}

\begin{remark}
For concrete calculations we will only be using the weaker vanishing result of Corollary \ref{main2}. The reason for this is that there is a simple criterion given by Kobayashi to check discrete decomposability, while we are not aware of a simple criterion to check if $\text{Hom}_{\f{g}',K'}(V_\pi,\c)=0$. The only independent application of Theorem \ref{main1}, that we know of, is the result that the non-$G$-invariant part of the Poincar\'e dual of a totally geodesic submanifold of $\G\b \text{SO}_0(2,p)/\text{SO}(2)\times SO(p)$, of the form $\G'\b\text{SO}_0(1,p)/\text{SO}(p)$, has Hodge type $(p,0)$ or $(0,p)$. In fact one can specify two $\theta$-stable Borel subalgebras $\f{b}_1$ and $\f{b}_2$, such that the cohomology class must have a non-zero $A_{\f{b}_i}$-component, for at least one $i=1,2$. This is shown in \cite{kob09}. It will be interesting to see more such direct applications of Theorem \ref{main1}. 
\end{remark}
% \begin{theorem}
% Suppose $R^+(\f{q})=R^-(\f{q})\le q$. 

% \end{theorem}

\section{Proof of Theorem~\ref{main}}\label{proof}

We wish to apply the aforementioned vanishing results to concrete examples of compact Hermitian locally symmetric spaces associated to simple Lie groups. We say $(\f{g}_0,\f{g}'_0)$ is a \emph{Hermitian symmetric pair} if the centre $Z(\f{k}_0)$ of $\f{k}_0$ is non-trivial and $Z(\f{k}_0)\subset\f{g}'_0$. For applying Lemma \ref{comana}, we must have that $(\f{g}_0,\f{g}'_0)$ is a Hermitian symmetric pair. On the other hand for applying Corollary \ref{main2}, we must know for which $\theta$-stable parabolic subalgebras $\f{q}$, $A_\f{q}$ is discretely decomposable as a $(\f{g}',K')$-module. This has already been done by Kobayashi and Oshima in \cite{kobosh} based on the criterion given by Kobayashi in \cite{kob94},\cite{kob97}. From now on we assume that $(\f{g}_0,\f{g}'_0)$ is a Hermitian symmetric pair. In this case $\f{p}=\f{p}^++\f{p}^-$, where $\f{p}^+$ and $\f{p}^-$ are the $+i$ and $-i$ eigenspaces of the complex structure $J$. A $\theta$-stable parabolic subalgebra $\f{q}$ is called \emph{holomorphic} or \emph{anti-holomorphic} if $\f{p}^+\subset\f{q}$ or $\f{p}^-\subset\f{q}$, respectively. By \cite[Thoerem 4.1 (4)]{kobosh}, $A_\f{q}$ is discretely decomposable as a $(\f{g}',K')$-module if $\f{q}$ is holomorphic or anti-holomorphic. Note that $\f{q}$ is holomorphic or anti-holomorphic then the pair $(R^+(\f{q}),R^-(\f{q}))$ is of the form $(p,0)$ or $(0,q)$, respectively. Hence both Lemma \ref{comana} and Corollary \ref{main2} applies. For the cases where $\f{q}$ is neither holomorphic nor anti-holomorphic, a list is given in \cite[Tables C.3, C.4]{kobosh}, of triples $(\f{g}_0,\f{g}'_0,\f{q})$ which satisfy that $A_\f{q}$ is discretely decomposable as a $(\f{g}',K')$-module. (Strictly speaking, this list includes, rather than consists of, all the cases where $\f{q}$ is neither holomorphic or anti-holomorphic.)
Let $(\f{g}_0,\f{g}'_0)$ be a pair and $D$ be the family of $\theta$-stable parabolic subalgebras $\f{q}$ of $\f{g}_0$, such that $A_\f{q}$ is discretely decomposable as a $(\f{g}',K')$-module. Let $G$ be a linear Lie group with Lie algebra $\f{g}_0$. By Theorem \ref{con}, any involution $\sigma$ of $G$ (and hence of $\f{g}_0$) induces an involution of $X_\G=\G\b G/K$, for some maximal compact $K$ and some torsion-free uniform lattice $\G$ in $G$, and the Poincar\'e dual $\mathcal{P}(Y_\G)$ of the resulting  fixed point submanifold $Y_\G$ is not $G$-invariant.  Moreover the fixed point submanifold of the involution of $X_\G$ induced by $\sigma\theta$ is complementary dimensional to $Y_\G$ and its Poincar\'e dual is again $G$-invariant. 
%By Theorem \ref{con}, any involution of $G$ (and hence $\f{g}_0$) induces an involution of $X_\G=\G\b G/K$, for some maximal compact $K$ and some uniform lattice $\G$ in $G$. For each row in Table \ref{list}, we will calculate
Let $t$ be the minimum of the complex dimensions of the classes $\mathcal{P}(Y_\G)$ and its complement.
%Call the minimum of these numbers $t$.
Let $Q$ be the set of all $\theta$-stable parabolic subalgebras $\f{q}$, which satisfy $R^+(\f{q})=R^-(\f{q})\le t$. Then by Lemma \ref{comana} and Corollary \ref{main2} we can conclude that $\mathcal{P}(Y_\G)$ has a non-zero $A_\f{q}$-component for some $\f{q}\in Q\setminus D$. We are only interested in the cases where $Q\setminus D$ is singleton, so that we get a precise non-vanishing result. This is the idea of proof of Theorem \ref{main}.
%Of the triples listed in \cite[Tables C.3, C.4]{kobosh}, such that $(\f{g}_0,\f{g}_0)$ is Hermitian symmetric, the  $Q\setminus D$

%One can identify 
%In Table \ref{list}, we reproduce (with slight changes of notation)
%the part of the list for which $(\f{g}_0,\f{g}'_0)$ is Hermitian symmetric. 
Now let us explain some notations. Let $\f{t}_0$ be a maximal abelian subalgebra of $\f{k}_0$. Fix a subset  $\Phi_\f{k}^+\subset\Phi_\f{k}$ of positive roots of $(\f{k},\f{t})$. Let $\Phi_n$ denote the set of weights of the adjoint representation of $\f{k}$ on $\f{p}$. Let $\Phi:=\Phi_\f{k}\cup\Phi_n$. For each $\alpha\in\Phi$ let $\f{g}_\alpha$ denote the corresponding root or weigh space. Then any $\theta$-stable parabolic subalgebra (up to the equivalence $\f{q}\sim\f{q}'$ if $A_\f{q}\cong A_{\f{q}'}$) is of the form
$\f{q}_\lambda=\bigoplus_{\langle\lambda,\alpha\rangle\ge 0}\f{g}_\alpha,$
where $\lambda$ is in some Euclidean space $\mathbb{E}$ containing $(i\f{t}_0)^*$ as a sub-Euclidean space and $\lambda$ is \emph{dominant} with respect to $\Phi_\f{k}^+$. We could have taken $\mathbb{E}$ to be equal to $(i\f{t}_0)^*$, but we allow a larger space for ease of calculation. The $\theta$-stable parabolic subalgebras will be identified by this parameter $\lambda$. %Since $\f{g}_0$ is associated to irreducible Hermitian symmetric spaces, $\f{t}$ is in fact a Cartan subalgebra of $\f{g}$ and hence all the $\alpha$ are roots of $(\f{g},\f{t})$. For each
If we fix a basis $(\epsilon_1,\cdots,\epsilon_m)$ of $\mathbb{E}$, then $\lambda$ can be written as $\sum_{i=1}^m a_i\epsilon_i$. For $\f{g}_0=\f{su}(p,q)$, we now indicate the Euclidean space $\mathbb{E}$, a set of positive roots of $(\f{k},\f{t})$ and the set of weights of $\f{p}$.

\vspace{.2cm}

    \noindent$\f{g}_0=\f{su}(p,q)$.
    
    \vspace{.1cm}
    
    \noindent$\mathbb{E}:=\r^{p+q}$ with standard orthonormal basis $(\epsilon_1,\cdots,\epsilon_{p+q})$.
    
    \vspace{.1cm}
    
    \noindent$\Phi_\f{k}^+=\{\epsilon_i-\epsilon_j:1\le i<j\le p\text{ and }p+1\le i<j\le p+q\}$.
    
    \vspace{.1cm}
    
    \noindent$\Phi_n=\{\pm(\epsilon_i-\epsilon_j):1\le i\le p\text{ and }p+1\le j\le p+q\}$.
    
    \vspace{.1cm}
    
    \noindent Dominant condition: $a_1\ge\cdots\ge a_p$ and $a_{p+1}\ge\cdots\ge a_{p+q}$.
    
\vspace{.2cm}

In Table \ref{list} we indicate the pairs $(\f{g}_0,\f{g}'_0)$, where $\f{g}_0=\f{su}(p,q)$, and the coordinates of $\lambda$, such that $A_{\f{q}_\lambda}$ is discretely decomposable as a $(\f{g}',K')$-module. This is taken from \cite[Table C.3]{kobosh}.
%, but one can easily check the discrete decomposability for these modules directly from the simple criterion given by Kobayashi \cite[Theorem 3.2]{kob94}. First note that this if this criterion is satisfied by for some $\f{q}'$, then it is also satisfied by all $\f{q}\supset \f{q}'$, \cite[Lemma 2.10]{kobosh}. Then note that the family of $\f{q}_\lambda$, indicated by the coordinates of $\lambda$ in each row, contain the $\theta$-stable Borel subalgebra given by coordinates 

 \begin{table}[ht]
    \centering
    
    \begin{tabular}{|c|c|}
    \hline
     Hermitian symmetric pairs & coordinates of $\lambda$ \\
    \hline
    \hline
     $(\f{su}(p,q),\f{su}(k)\oplus\f{su}(p-k,q)\oplus\f{u}(1)),$ & $a_p\ge a_{p+1}$,\\
     $q\ge p>k$ & $(\exists p+1\le l\le p+q-1)$  \\
     &$ a_l\ge a_1$ and $a_p\ge a_{l+1}$, \\
     & or $a_{p+q}\ge a_1$\\
    \hline 
     ($\f{su}(p,q),\f{su}(p,q-k)\oplus\f{su}(k)\oplus\f{u}(1)),$ & $a_{p+q}\ge a_1$,\\
     $q\ge p, k< q$ & $(\exists 1\le l\le p-1)$ $ a_l\ge a_{p+1}$ \\
     & and $a_{p+q}\ge a_{l+1}$, \\
     & or $a_p\ge a_{p+1}$\\
    \hline
    % $(\f{su}(n,n),\f{so}^*(2n))$ & & \\
    % \hline
    % $(\f{su}(n,n),\f{sp}(n,\r))$ & & \\
    % \hline
    % 3 & $(\f{so}(2,2n),\f{so}(2,k)\oplus\f{so}(2n-k))$ & $|a_1|\ge|a_2|$ \\
    % \hline
    % 4 & $(\f{so}(2,2n+1),\f{so}(2,k)\oplus\f{so}(2n+1-k))$  & $|a_1|\ge a_2$ \\
    % \hline
    % 5 & $(\f{so}(2,4),\f{u}(1,2)_1)$  & $-a_3\ge|a_1|$ \\
    % \hline
    % 6 & $(\f{so}(2,4),\f{u}(1,2)_2)$  & $a_3\ge|a_1|$ \\
    %  \hline
    % 7 & $(\f{so}(2,2n),\f{u}(1,n))$  & $(a_1,a_2,a_3\cdots,a_{n+1})$\\
    % &&$=(s,0,0,\cdots,0)$\\
    % && or $(0,s,0,\cdots,0)$\\
    % \hline
    % 8 & $(\f{so}^*(2n),\f{so}^*(2n-2)\oplus\f{so}^*(2))$ & $(a_1,\cdots,a_n)=(\underbrace{s,\cdots,s}_k,-s,$\\
    % && $\cdots,-s) (\exists 1\le k\le n-1)$\\
    % \hline
    % 9 & $(\f{so}^*(2n),\f{u}(n-1,1))$  & $(a_1,\cdots,a_n)=(\underbrace{s,\cdots,s}_k,-s,$\\
    % && $\cdots,-s) (\exists 1\le k\le n-1)$\\
    % \hline
    % % $\f{sp}(n,\r)$ & & \\
    % % \hline
    % 10 & $(\f{e}_{6(-14)},\f{so}(2,8)\oplus\f{so}(2))$  & $(a_1,\cdots,a_8)=(s,\cdots,s,-s)$\\
    % && or $(s,s,s,s,-s,-s,-s,s)$\\
    % \hline
    % $\f{e}_{7(-25)}$ & & \\
    % \hline
\end{tabular}
\caption{Discretely decomposable $A_{\f{q}_\lambda}$ when $\f{g}_0=\f{su}(p,q)$}
    \label{list}
\end{table}

Now we are ready to implement the idea described in the beginning of this section to some of the cases given in Table \ref{list}. 

\begin{proof}\emph{of Theorem \ref{main}.}
%  For each of the rows 1 (with $k=1, 5\le p\le q\le 2q-2$), 2 (with $k=1, 5\le p\le q, p\ne q-1$), 5 and 6 in Table \ref{list}, we will show that $Q\setminus D$ is singleton and the sole members of $Q\setminus D$ in each case are the $\f{q}_\lambda$ described in the statement of the theorem. As already noted above, for rows 5 and 6, the set $Q$ consists of all $\theta$-stable parabolic subalgebras satisfying $R^+(\f{q})=R^-(\f{q})$. Using the classification described above, $Q=\{\f{q}_{\lambda_2},\f{q}_{\lambda_3},\f{q}_\mu\}$. Then $Q\cap D=\{\f{q}_{\lambda_2},\f{q}_\mu\}$ for row 5 and $Q\cap D=\{\f{q}_{\lambda_2},\f{q}_{\lambda_3}\}$ for row 6. Thus for both the rows we have $|Q\setminus D|=1$.
We retain the notations used before in this section. We consider the pairs in Table \ref{list}, when $k=1$. We first deal with the pair in the first row of Table \ref{list}, with $k=1$. In this case $t=q$. 
%As explained before we have to  For both rows in Table \ref{list}, with $k=1$, we will show that $Q\setminus D$ is singleton and the sole members of $Q\setminus D$ in each case are the $\f{q}_\lambda$ described in the statement of the theorem.  Let us analyze row 1 , when $k=1$. Then $t=q$.  
Let $\lambda=\sum_{i=1}^{p+q}a_i\epsilon_i$ be such that $\f{q}_\lambda\in Q\setminus D$. Then we must have $R^+(\f{q}_\lambda)=R^-(\f{q}_\lambda)\le q$ and 
\begin{equation}\label{ineq}
    a_1>a_s>a_p,\text{ for some }p+1\le s\le p+q.
\end{equation}
Fix $s$. Let $x,y,z,l,m,n\in\n$ satisfying $x+y+z=p$ and $l+m+n=q$, such that, $a_i>a_s$ for $1\le i\le x$ and $p+1\le i\le p+l$, $a_i=a_s$ for $x+1\le i\le x+y$ and $p+l+1\le i\le p+l+m$, $a_s>a_i$ for $x+y+1\le i\le p$ and $p+l+m+1\le i\le p+q$. Note that by (\ref{ineq}), $x,z,m\ge 1$. Then we have
\begin{align}
q\ge R^+(\f{q}_\lambda)&\ge x(m+n)+yn=n(x+y)+mx\label{ineq1}\\ 
q\ge R^-(\f{q}_\lambda)&\ge z(l+m)+yl=l(y+z)+mz\label{ineq2}\\ 
2q\ge R(\f{q}_\lambda)&\ge
x(m+n)+y(n+l)+z(l+m)\label{ineq3}
%&=l(y+z)+m(z+x)+n(x+y)
\end{align}
From (\ref{ineq3}) we get that at least one of $m+n,n+l$ and $l+m$ is less than or equal to $2$. Let us assume $l+m\le 2$. We will show a contradiction. Since $n\ge q-2$, (\ref{ineq1}) implies $R^+(\f{q}_\lambda)\ge (q-2)(x+y)+xm$. If $x+y\ge 2$ then $R^+(\f{q}_\lambda)\ge 2(q-2)+1>q$, since $q\ge 5$. Thus $x+y\le 1$, which implies $x=1$ and $y=0$. Hence $z=p-1$. From (\ref{ineq2}) we get $R^-(\f{q}_\lambda)\ge (p-1)(l+m)+ly$. If $l+m=2$ then $R^-(\f{q}_\lambda)\ge 2(p-1)>q$, by assumption. Thus $l+m\le 1$ which implies $l=0$ and $m=1$. Hence $n=q-1$. This implies $R^+(\f{q}_\lambda)\ge q$ and $R^-(\f{q}_\lambda)\ge p-1$. But then $R^+(\f{q}_\lambda)= q$. This implies $a_1>a_{p+1}>a_{p+2}\ge\cdots\ge a_{p+q}\ge a_2\ge\cdots\ge a_p$. The inequality $a_{p+q}\ge a_2$ must be strict since otherwise $R^-(\f{q})=(p-1)q>q$. Hence  $R^-(\f{q}_\lambda)=p-1$. This is a contradiction since $R^+(\f{q}_\lambda)\ne R^-(\f{q}_\lambda)$. Thus $l+m>2$. Similarly $m+n>2$. Thus we must have $l+n\le 2$. So $m\ge q-2$. Then (\ref{ineq1}) and (\ref{ineq2}) implies that $R^+(\f{q}_\lambda)\ge (q-2)x$ and $R^-(\f{q}_\lambda)\ge (q-2)z$. Since $q\ge 5$ we must have $x=1=z$. Hence $y=p-2$. If $l\ge 1$ then (\ref{ineq2}) implies that $R^-(\f{q}_\lambda)\ge p-1+q-2>q$, which is a contradiction. Hence $l=0$. Similarly (\ref{ineq1}) implies $n=0$. Thus $m=q$ and we have
\begin{equation}\label{eq}
    a_1>a_2=\cdots=a_{p-1}=a_{p+1}=\cdots=a_{p+q}>a_p.
\end{equation}
All the parameters $\lambda=\sum_{i=1}^{p+q}a_i$, satisfying (\ref{eq}), give the same $\theta$-stable parabolic subalgebra $\f{q}_\lambda$. It satisfies   $R^+(\f{q}_\lambda)=q=R^-(\f{q}_\lambda)$. Thus $Q\setminus D$ is indeed singleton.

For the pair in the second row of Table \ref{list}, with $k=1$, we repeat the above calculation with $p$ and $q$ interchanged. This time $l+m\le 2$ or $m+n\le 2$ does not lead to a contradiction when $q=p+1$. In that case we have $|Q\setminus D|=3$. But if we assume $q=p$ or $q>p+1$, then again we get that $Q\setminus D$ is a singleton. 
%This finishes the proof. 
\end{proof}

\begin{remark}
1. As a part of \cite[Theorem 1.1]{monsan} it was already proved that, when $G=\text{SU}(p,q)$ with $p<q-1, q\ge 5$, there exists a uniform lattice $\G$ and a geometric class in $H^*(\G\b \text{SU}(p,q)/\text{S}(\text{U}(p)\times \text{U}(p))$ having an $A_{\f{q}_\lambda}$-component, where $\lambda=\epsilon_{p+1}-\epsilon_{p+q}$.

%It also follows from Li \cite{li}. But his techniques are different.
\vspace{.1cm}

\noindent 2.  For all Hermitian symmetric pairs $(\f{g}_0,\f{g}'_0)$ listed in the tables given in \cite{kobosh}, except the ones in Table \ref{list}, either the set $Q\setminus D$ contains more than one elements or the set $Q$ itself is singleton. In the first case we don't get a precise non-vanishing result while all examples in the second case has already been studied in \cite{monsan}.

\vspace{.1cm}

\noindent 3. The statement of Theorem \ref{main} also holds when $\f{g}_0=\f{su}(2,2)$. It is rather easy to prove using our method. 
\end{remark}

\noindent\textbf{Acknowledgements:} While preparing this paper the author had the good fortune of meeting Toshiyuki Kobayashi in a conference. He pointed out his paper with Y. Oshima \cite{kobosh}, which is vitally used in this work. On a later visit to him, he also shared his 2009 preprint \cite{kob09}, which has a beautiful direct application of Theorem \ref{main1}. The author is also grateful to Parameswaran Sankaran for many helpful comments.

\end{document}